\documentclass[11 pt]{amsart}

\usepackage[a4 paper, margin=3.3cm]{geometry}
\usepackage{amscd}
\usepackage{amsmath}
\usepackage{amsthm}
\usepackage{amsfonts}
\usepackage{amssymb}
\usepackage{amsxtra}

\def\C{\mathbb{C}}

\def\N{\mathbb{N}}

\def\R{\mathbb{R}}

 \newtheorem{thm}{Theorem}[section]
 
 \newtheorem{lem}[thm]{Lemma}

 \newtheorem{rem}[thm]{Remark}

\newcommand{\be}{\begin{equation}}
\newcommand{\ee}{\end{equation}}
\newcommand{\bea}{\begin{eqnarray}}

\newcommand{\eea}{\end{eqnarray}}
\newcommand{\Bea}{\begin{eqnarray*}}
\newcommand{\Eea}{\end{eqnarray*}}

\newcounter{cnt1}
\newcounter{cnt2}
\newcounter{cnt3}
\newcommand{\blr}{\begin{list}{$($\roman{cnt1}$)$}
 {\usecounter{cnt1} \setlength{\topsep}{0pt}
 \setlength{\itemsep}{0pt}}}
\newcommand{\bla}{\begin{list}{$($\alph{cnt2}$)$}
 {\usecounter{cnt2} \setlength{\topsep}{0pt}
 \setlength{\itemsep}{0pt}}}
\newcommand{\bln}{\begin{list}{$($\arabic{cnt3}$)$}
 {\usecounter{cnt3} \setlength{\topsep}{0pt}
 \setlength{\itemsep}{0pt}}}
\newcommand{\el}{\end{list}}
\title[HUP for n-parallel lines]{Heisenberg uniqueness pairs corresponding to a finite number of parallel lines}

\author{Sayan Bagchi}
\date{}
\begin{document}


\address{Stat-Math Unit,
Indian Statistical Institute, Kolkata - 700102, India}
\email{sayansamrat@gmail.com}




\begin{abstract}
In this paper, we study the Heisenberg uniqueness pairs corresponding to a finite number of parallel lines $\Gamma$. We give a necessary condition and a sufficient condition for a subset $\Lambda$ of $\R^2$ so that $(\Gamma,\Lambda)$ becomes a HUP.

\end{abstract}


\maketitle

\section[Introduction]{Introduction} Let $\Gamma$ be a curve on $\R^2$ and $\Lambda$ be a subset of $\R^2$. We call $(\Gamma, \Lambda)$ to be a Heisenberg uniqueness pair if
$$\hat{\mu}|_{\Lambda}=0 \mbox{ implies } \mu=0$$
for every measure $ \mu$ on $\R^2$ which is supported on $\Gamma$ and also absolutely continuous with respect to the arc length of the curve. Here $\hat{\mu}$ stands for the Fourier transform of $\mu$ defined by
$$\hat{\mu}(\xi,\eta)=\int_{\R^2}e^{\pi i (x\xi+y\eta)}d\mu(x,y)$$
where $\xi, \eta\in \R^2$.

The terminology "Heisenberg uniqueness pair" was first introduced by Hedenmalm and Montes-Rodr\'{\i}guez in \cite{HM}. In that paper they also studied the cases where $\Gamma$ are a hyperbola and two parallel straight lines. After that, Heisenberg uniqueness pairs were studied for some more well-known curves on $\R^2$ by various people. For example, the case of circle and parabola were studied by Nir Lev \cite{L} and Per Sj\"{o}lin \cite{S2} respectively. Please see \cite{GS}, \cite{JK}, \cite{S1}, \cite{V1} and \cite{V2} for more results in this direction.

 In this paper we extend the results of \cite{B} to finite number of parallel lines. However, unlike \cite{B}, we are not able to characterize Hiesenberg pairs in the case $\Gamma$ is a set of finitely many parallel lines. 
\section{Main Result}
In order to state our result we first set up some notations. Consider a subset $E$ of $\R$ and a point $\xi\in E$. Then we can define the following sets:\\
For $2\leq l\leq n$, $P^{E, \xi}_l=\{ \psi: E\rightarrow \C$ : there is an interval $I_\xi$ around $\xi$ and functions $\varphi_j\in L^1(\R)$, $j=1,2, ...,l-1$ such that
$$\hat{\varphi_1}+ \hat{\varphi_2}\psi+\hat{\varphi_3}\psi^2+\cdots + \hat{\varphi_{l-1}}\psi^{l-2}+\psi^{l-1}=0$$

on $I_\xi\cap E\}$. It is worthy to note the particular case $l=2$, where $P^{E, \xi}_2=\{ \psi: E\rightarrow \C$ : there is an interval $I_\xi$ and functions $\varphi\in L^1(\R)$, such that
$\hat{\varphi}= \psi$

on $I_\xi\cap E\}$. We have to work with this particular case many times throughout the paper. Note that by Wiener's lemma (\cite{K}, page 57), if $\psi\in P^{E, \xi}_2$, then $\frac{1}{\psi}\in P^{E, \xi}_2$.

Let us consider $n$ parallel lines $\Gamma=\R\times \{\alpha_1, \alpha_2, \cdots, \alpha_n\}$ where $\alpha_{i+1}-\alpha_i=constant$, $1\leq i \leq n-1$. By translation and rescaling discussed in \cite{HM}, it is enough to consider $\Gamma=\R\times \{0,1,2,\cdots, n-1\}$. If $\mu$ is a measure such that $\mu$ is supported on $\Gamma$ and absolutely continuous with respect to arc length of $\Gamma$, then there exists functions $f_0, f_1, \cdots f_{n-1}\in L^1(\R)$, such that
\be\label{eq:mu}\hat{\mu}(\xi, \eta)=\hat{f_0}(\xi)+e^{\pi i\eta} \hat{f_1}(\xi)+ e^{2\pi i\eta}\hat{f_2}(\xi)+\cdots + e^{(n-1)\pi i \eta}\hat{f}_{n-1}(\xi).\ee
Since $\hat{\mu}$ is 2-periodic with respect to the second variable, we may assume $\Lambda$ to be 2-periodic with respect to the second coordinate.

Let $\Pi(\Lambda)$ be the projection of $\Lambda$ on $x$-axis. That is,
$$\Pi(\Lambda)=\{\xi\in \R : (\xi, \eta)\in \Lambda \mbox{ for some } \eta\in \R\}.$$
Also consider the set,
$$Img(\xi)=\{\eta\in \R:(\xi, \eta)\in \Lambda, 0\leq \eta < 2\}. $$
We can write $\Pi(\Lambda)$ as a union of $n$ number of disjoint sets. The disjoint sets are described below.
\begin{enumerate}
\item For $1\leq k <n$, $\Pi^k(\Lambda)= \{\xi\in \Pi(\Lambda):$ there are exactly k number of distinct points $\eta_1, \eta_2,\dots, \eta_k\in Img(\xi)\}$. We can assume $\eta_1< \eta_2<\cdots <\eta_k$ for each $\xi\in\Pi^k(\Lambda)$.\\
\item $\Pi^n(\Lambda)= \{\xi\in \Pi(\Lambda): \mbox{ there are atleast n number of distinct points in } Img(\xi)\}$.
\end{enumerate}
Similar as \cite{B}, here also we have to find a subset for each $\Pi^k(\Lambda)$, $1\leq k\leq n-1$, which will be removed from $\Pi^k(\Lambda)$ in order to get our result. Before defining those sets let us set up some more notations.
For each $\xi \in \Pi^k(\Lambda)$, $1<k<n$ we will consider two matrices as follows
$$A_{\xi,n-1}^{\Pi^k(\Lambda)}=\begin{bmatrix} 1 & e^{\pi i \eta_1} & e^{2\pi i \eta_1} & \dots & e^{(n-2)\pi i \eta_1} \\ 1 & e^{\pi i \eta_2} & e^{2\pi i \eta_2} & \dots & e^{(n-2)\pi i \eta_2}\\ \hdotsfor{5} \\ 1 & e^{\pi i \eta_k} & e^{2\pi i \eta_k}  & \dots & e^{(n-2)\pi i \eta_k} \end{bmatrix} \mbox{ and }
A_{\xi,k}^{\Pi^k(\Lambda)}= \begin{bmatrix} 1 & e^{\pi i \eta_1} & e^{2\pi i \eta_1} & \dots & e^{(k-1)\pi i \eta_1} \\ 1 & e^{\pi i \eta_2} & e^{2\pi i \eta_2} & \dots & e^{(k-1)\pi i \eta_2}\\ \hdotsfor{5} \\ 1 & e^{\pi i \eta_k} & e^{2\pi i \eta_k}  & \dots & e^{(k-1)\pi i \eta_k} \end{bmatrix} $$
where $\eta_i\in Img(\xi)$, $i=1,...,k$. Notice that for $k=n-1$, both the matrices are same. Let us also consider two column vectors associated to each $\xi\in \Pi^k(\Lambda)$, $k=2,...,n-1$, defined as
$$B_{\xi, n-1}^{\Pi^k(\Lambda)}=\left[ \begin{array}{cc} -e^{(n-1)\pi i \eta_1} \\ -e^{(n-1)\pi i \eta_2}\\.\\.\\-e^{(n-1)\pi i \eta_k} \end{array} \right] \mbox{ and }
B_{\xi, k}^{\Pi^k(\Lambda)}=\left[ \begin{array}{cc} -e^{k\pi i \eta_1} \\ -e^{k\pi i \eta_2}\\.\\.\\-e^{k\pi i \eta_k} \end{array} \right]$$
Here also, two column vectors are same when $k=n-1$. Let us now define the following sets:
\begin{enumerate}
\item $\Pi^{1*}(\Lambda)=\{\xi\in \Pi^1(\Lambda):\chi\in P^{\Pi^1(\Lambda), \xi}_n,$ where $\chi(\xi)= e^{\pi i \eta}, \eta\in Img(\xi)$\}\\
\item For $1<k<n-1$, $\Pi^{k*}_{su}(\Lambda)=\{\xi\in \Pi^k(\Lambda):$ there exists an interval $I_\xi$ around $\xi$ and $\varphi_i\in L^1(\R), i=1,2,...,n-1$ such that $(\hat{\varphi_1},\hat{\varphi_2},...,\hat{\varphi_{n-1}})$ is a solution of the system of equation $A_{\xi,n-1}^{\Pi^k(\Lambda)}X_\xi= B_{\xi, n-1}^{\Pi^k(\Lambda)}$ on $I_\xi\cap \Pi^k(\Lambda)\}$,\\
\item For $1<k<n-1$,$\Pi^{k*}_{ne}(\Lambda)=\{\xi\in \Pi^k(\Lambda)$: there exists an interval $I_\xi$ around $\xi$ and $\varphi_i\in L^1(\R), i=1,2,...,k $ such that $(\hat{\varphi_1},\hat{\varphi_2},...,\hat{\varphi_k})$ is the solution of the system of equation $A_{\xi,k}^{\Pi^k(\Lambda)}X_\xi= B_{\xi, k}^{\Pi^k(\Lambda)}$, on $I_\xi\cap \Pi^k(\Lambda)\}, $\\
\item For $k=n-1$, $ \Pi^{n-1*}(\Lambda)= \{\xi\in \Pi^{n-1}(\Lambda):$ there exists an interval $I_\xi$ around $\xi$ and $\varphi_i\in L^1(\R), i=1,2,...,n-1$ such that $(\hat{\varphi_1},\hat{\varphi_2},...,\hat{\varphi_{n-1}})$ is the solution of the system of equation $A_{\xi,n-1}^{\Pi^{n-1}(\Lambda)}X_\xi= B_{\xi, n-1}^{\Pi^{n-1}(\Lambda)}$, on $I_\xi\cap \Pi^{n-1}(\Lambda)\}$
\end{enumerate}

\begin{rem} One can easily see that each of the sets defined above will remain same if we change the order of the points in $Img(\xi)$.
\end{rem}

Now we are in a position to state our main result.
\begin{thm}\label{thm:main}
Let $\Gamma=\R\times \{0, 1, 2,\dots, n-1\}$ where $n\in \N$ and $\Lambda\in\R^2$ be a closed set which is 2-periodic with respect to the second variable. If $(\Gamma, \Lambda)$ is a HUP , then
$$\widetilde{\Pi_{ne}(\Lambda)}= \Pi^n(\Lambda)\bigcup\left(\Pi^1(\Lambda)\setminus \Pi^{1*}(\Lambda)\right)\bigcup \left(\Pi^{n-1}(\Lambda)\setminus \Pi^{n-1*}(\Lambda)\right)\bigcup_{j=2}^{n-2}\left(\Pi^j(\Lambda)\setminus \Pi_{ne}^{j*}(\Lambda)\right)$$
is dense in $\R$. Conversely, if
$$\widetilde{\Pi_{su}(\Lambda)}= \Pi^n(\Lambda)\bigcup\left(\Pi^1(\Lambda)\setminus \Pi^{1*}(\Lambda)\right)\bigcup \left(\Pi^{n-1}(\Lambda)\setminus \Pi^{n-1*}(\Lambda)\right)\bigcup_{j=2}^{n-2}\left(\Pi^j(\Lambda)\setminus \Pi_{su}^{j*}(\Lambda)\right)$$
is dense in $\R$, then $(\Gamma, \Lambda)$ is a HUP.
\end{thm}

Note that, when $n=3$, the above theorem gives the same result proved by Babot in \cite{B}.
\section{Proof of the main theorem}
We will start this section by proving that $\Pi^{k*}_{ne}(\Lambda)$ is actually contained in $\Pi^{k*}_{su}(\Lambda)$.
\begin{lem}\label{lem:subset}

\begin{enumerate}

\item $\Pi^{k*}_{ne}(\Lambda)\subset \Pi^{k*}_{su}(\Lambda),\;\;\; k=2,3, \cdots , n-2.$\\
\item $P^{E, \xi}_2\subset P^{E, \xi}_3\subset ... \subset P^{E, \xi}_n.$
\end{enumerate}
\end{lem}
\begin{proof}(1) Let $\xi_0\in \Pi^{k*}_{ne}(\Lambda)$. Then there exists an interval $I_{\xi_0}$ containing $\xi_0$ and functions $\varphi_i \in L^1(\R)$, $i=1,2,,\cdots,k$ such that
\be\label{eq:k}\hat{\varphi_1}(\xi)+e^{\pi i \eta_j}\hat{\varphi_2}(\xi)+\cdots +e^{(k-1)\pi i \eta_j}\hat{\varphi_k}(\xi)=-e^{k\pi i \eta_j}\ee
on $I_{\xi_0}\cap \Pi^{k}(\Lambda)$, for $j=1,2,...,k$. Here, $\eta_1, \eta_2,...,\eta_k\in Img(\xi)$.

Now, let $f$ be a function in $L^1(\R)$ such that the support of $\hat{f}$ is contained in $I_{\xi}$. Multiplying both side of the equation (\ref{eq:k}) by $e^{\pi i \eta_j}$ and $\hat{f}$ respectively and then adding we get
$$\hat{f}\hat{\varphi_1}+e^{\pi i \eta_j}(\hat{\varphi_1}+\hat{f}\hat{\varphi_2})+\cdots+ e^{(k-1)\pi i \eta_j}(\hat{\varphi_{k-1}}+\hat{f}\hat{\varphi_{k}})+e^{k \pi i \eta_j}(\hat{\varphi_k}+\hat{f})=-e^{(k+1) \pi i \eta_j}$$
on $I_{\xi_0}\cap \Pi^k(\Lambda)$. Continuing this process iteratively we will find $\phi_i\in L^1(\R)$, $i=1,2,\cdots, n-1$ such that
$$\hat{\phi_1}(\xi)+e^{\pi i \eta_j}\hat{\phi_2}(\xi)+\cdots +e^{(n-2)\pi i \eta_j}\hat{\phi_{n-1}}(\xi)=-e^{(n-1)\pi i \eta_j}$$
on $I_{\xi_0}\cap \Pi^{k}(\Lambda)$ for $j=1,2,\cdots, k$. This implies $\xi_0\in \Pi^{k*}_{su}(\Lambda)$. Hence, first part of the lemma is proved.

(2) Suppose $\psi\in P^{\Pi^1(\Lambda), \xi}_2$. Then there exists an interval $I_{\xi}$ around $\xi$ and a function $\varphi\in L^1(\R)$ such that $\psi=\hat{\varphi}$. Choosing $\varphi_1= -\frac{1}{2}\varphi$ and $\varphi_2=-\frac{1}{2}\varphi*\varphi$ one can easily see that
$$\hat{\varphi_2}+\hat{\varphi_1}\psi+\psi^2=0$$
on $I_{\xi}\cap\Pi^1(\Lambda)$. Hence, $P^{\Pi^1(\Lambda), \xi}_2\subset P^{\Pi^1(\Lambda), \xi}_3$. Now we have to prove $P^{\Pi^1(\Lambda), \xi}_l\subset P^{\Pi^1(\Lambda), \xi}_{l+1}$, for $l>2$. We can prove that by using similar technique used to prove first part of the lemma.
\end{proof}

\subsection{Proof of the sufficient part}: Let $\widetilde{\Pi_{su}(\Lambda)}$ be dense in $\R$. Also, suppose $\hat{\mu}|_{\Lambda}=0$ for some measure $\mu$ which is supported on $\Gamma$ and absolutely continuous with respect to the arc length of $\Gamma$. This implies there exists $f_1, f_2,\cdots,f_n\in L^1(\R)$ such that
$$\hat{f}_1(\xi)+ e^{\pi i \eta} \hat{f}_2(\xi)+\cdots+e^{(n-1)\pi i \eta}\hat{f}_n(\xi)=0$$
for all $(\xi,\eta)\in\Lambda$. Since $\widetilde{\Pi_{su}(\Lambda)}$ is dense in $\R$, in order to prove $\mu=0$, it is enough to prove that $f_i|_{\widetilde{\Pi_{su}(\Lambda)}}=0$.

\underline{Case-I}: For $\xi\in \Pi^1(\Lambda)$, we have
\be\label{eq:suff1}\hat{f}_1(\xi)+ e^{\pi i \eta} \hat{f}_2(\xi)+\cdots+e^{(n-1)\pi i \eta}\hat{f}_n(\xi)=0,\ee
 where $\eta\in Img(\xi)$.  Let $\xi_0$ be a point in $\Pi^1(\Lambda)$. If $\hat{f}_n(\xi_0)\neq 0$, then by Wiener's lemma $\frac{1}{\hat{f}_n}\in P^{\Pi^1(\Lambda), \xi_0}_2$. So, by the definition of $P^{\Pi^1(\Lambda), \xi_0}_2$ there exists an interval $I_{\xi_0}$ around $\xi_0$ and $\varphi\in L^1(\R)$ such that $\hat{\varphi}=\frac{1}{\hat{f_n}}$ on $I_{\xi_0}\cap \Pi^1(\Lambda)$. Dividing Equation (\ref{eq:suff1}) by $\hat{f}_n(\xi)$ we have
$$\frac{\hat{f}_1(\xi)}{\hat{f}_n(\xi)}+ e^{\pi i \eta_0} \frac{\hat{f}_2(\xi)}{\hat{f}_n(\xi)}+\cdots+e^{(n-1)\pi i \eta_0}=0$$
on $I_{\xi_0}\cap \Pi^1(\Lambda)$.
Let $\chi(\xi)=e^{\pi i \eta}$, where $\eta\in Img(\xi)$. As $\frac{\hat{f}_i(\xi)}{\hat{f}_n(\xi)}\in P^{\Pi^1(\Lambda), \xi_0}_2$ for $i=1,2,\cdots, n-1$, we can conclude that $\chi\in P^{\Pi^1(\Lambda), \xi_0}_n$ and hence $\xi_0\in \Pi^{1*}(\Lambda).$

If $\hat{f}_l(\xi_0)\neq 0$ for some $2\leq l<n$ but $\hat{f}_i(\xi_0)=0$ for all $i>l$, then we can similarly show that $\chi_0 \in  P^{\Pi^1(\Lambda), \xi_0}_l$. Hence, applying Lemma \ref{lem:subset} we have $\xi_0\in \Pi^{1*}(\Lambda)$.

Thus we can conclude that if $\xi_0\in \Pi^{1}(\Lambda)\setminus \Pi ^{1*}(\Lambda)$, then $f_i(\xi_0)=0$ for all $i=1,2,\cdots , n$.

\underline{Case-II}: Choose $k$ so that $2\leq k< n-1$. For $\xi\in \Pi^k(\Lambda)$, we have
\be\label{eq:suff2}\hat{f}_1(\xi)+ e^{\pi i \eta_j} \hat{f}_2(\xi)+\cdots+e^{(n-1)\pi i \eta_j}\hat{f}_n(\xi)=0,\ee
where $\eta_j\in Img(\xi)$ for $j=1,2,\cdots, k$. Suppose, $\xi_0\in \Pi^k(\Lambda)$. Now if $\hat{f}_n (\xi_0)\neq 0$, then again by Wiener's lemma $\frac{1}{\hat{f}_n}\in P^{\Pi^1(\Lambda), \xi_0}_2$. Let $I_{\xi_0}$ be the corresponding interval around $I_{\xi_0}$. Dividing equation (\ref{eq:suff2}) by $\hat{f}_n(\xi)$ we have
$$\frac{\hat{f}_1(\xi)}{\hat{f}_n(\xi)}+ e^{\pi i \eta_j} \frac{\hat{f}_2(\xi)}{\hat{f}_n(\xi)}+\cdots+e^{(n-1)\pi i \eta_j}=0,$$
for all $\xi\in I_{\xi_0}\cap \Pi^k(\Lambda)$. This implies $(\frac{\hat{f}_1}{\hat{f}_n}, \frac{\hat{f}_2}{\hat{f}_n},\cdots, \frac{\hat{f}_{n-1}}{\hat{f}_n})$ is a solution of the system of equation $A_{\xi,n-1}^{\Pi^k(\Lambda)}X_\xi= B_{\xi, n-1}^{\Pi^k(\Lambda)} \mbox{ on } I_{\xi_0}\cap \Pi^k(\Lambda)$. As each $\frac{\hat{f}_i}{\hat{f}_n}\in P^{\Pi^k(\Lambda), \xi_0}_2$, for $i=1,2,\cdots,n-1$, we can conclude that $\xi_0\in \Pi^{k*}_{su}$.

If $\hat{f}_l(\xi_0)\neq 0$ for some $l\geq k$ and $\hat{f}_j(\xi_0)=0$ for all $j>l$, then we will have
$$\frac{\hat{f}_1(\xi)}{\hat{f}_l(\xi)}+ e^{\pi i \eta_j} \frac{\hat{f}_2(\xi)}{\hat{f}_l(\xi)}+\cdots+e^{(n-1)\pi i \eta_j}=0$$
on $I_{\xi_0}\cap \Pi^k(\Lambda)$ for some interval $I_{\xi}$. Similarly as we have done in the first part of the Lemma \ref{lem:subset}, by taking functions whose Fourier transform is supported on $I_{\xi_0}$ and then by iterative method we can find $\varphi_i\in L^1(\R)$, $i=1,2,\cdots, n$ so that $(\hat{\varphi}_1, \hat{\varphi}_2,\cdots,\hat{\varphi}_{n-1})$ will be a solution of the system of equation $A_{\xi,n-1}^{\Pi^k(\Lambda)}X_\xi= B_{\xi, n-1}^{\Pi^k(\Lambda)} \mbox{ on } I_{\xi_0}\cap \Pi^k(\Lambda)$. Hence $\xi_0\in \Pi^{k*}_{su}(\Lambda)$.

If $f_j(\xi_0)=0$ for all $j\geq k$, then the equation (\ref{eq:suff1}) turns out to be a homogenous system of equation with k number of equations where the number of variables is less than k and determinant of the corresponding matrix is non-zero. Hence, $f_j(\xi_0)=0$ for all $j=1,2, \cdots , n$.

\underline{Case-III}: Suppose $\xi_0\in \Pi^{n-1}(\Lambda)$. Then similarly as we have shown in case-II, we can prove that if $\hat{f}_n(\xi_0)\neq 0$ then $\xi_0\in \Pi^{n-1*}(\Lambda)$ and if $\hat{f}_n(\xi_0)=0$ then $f_i(\xi_0)=0$ for all $i=1,2,\cdots, n$.

\underline{Case-IV}: If $\xi_0\in \Pi^n(\Lambda)$, then we have a system of homogeneous equation with n equations and n variables where the determinant of the corresponding matrix is non-singular. Hence $\hat{f}_j(\xi_0)=0$ for all $j=1,2,\cdots,n$.

Therefore, from the above four cases we can conclude that $\hat{f}_j|_{\widetilde{\Pi}(\Lambda)}=0$, as we required.

\subsection{Proof of the necessary part}For convenience, in the proof of necessary part we will use the
notations $\Pi^{1*}_{ne}(\Lambda)$ and $\Pi^{n-1*}_{ne}(\Lambda)$ in place of $\Pi^{1*}(\Lambda)$ and $\Pi^{n-1*}(\Lambda)$ respectively.

 Let $(\Gamma, \Lambda)$ be a Heisenberg uniqueness pair for some $\Lambda\subset \R^2$.
We will show that $\widetilde{\Pi_{ne}(\Lambda)}$ is dense in $\R$. As, $(\Gamma, \Lambda)$ is a HUP, one can notice that $\Pi(\Lambda)$ has to be dense in $\R$. This is because we can always consider measures which are only supported on $\R\times \{0\}$ and then apply the result observed in \cite{HM} corresponding to a single line. Now, suppose $\widetilde{\Pi_{ne}(\Lambda)}$ is not dense in $\R$. Then, there exists an interval $I_0$ such that $\widetilde{\Pi_{ne}(\Lambda)}\cap I_0$ is empty. But as $\Pi(\Lambda)$ is dense in $\R$, $\Pi(\Lambda)\cap I_0$ is non-empty. So, we have
$$\Pi(\Lambda)\cap I_0=\left(\bigcup_{k=1}^{n-1}\Pi^{k*}_{ne}(\Lambda)\right)\cap I_0.$$

 In order to prove the necessary condition we need the following lemmas.
\begin{lem}\label{lem:density}
Let $I$ be an interval such that $I\cap \Pi^{k*}_{ne}(\Lambda)$ is dense in $I$ for some $k$ satisfying $2\leq k\leq n-1$. Then there exist a subinterval $J\subset I$ such that $J\cap \Pi(\Lambda)\subset \bigcup_{j=k}^n \Pi^j(\Lambda)$.
\end{lem}
\begin{proof} Choose some $k$ satisfying $2\leq k\leq n-1$. In $\Pi^{k}(\Lambda)$ we can define a function $\tau$ as follows
\be\label{eq:tau}\tau(\xi)= \prod_{p=2}^k\prod_{q=1}^{p-1}(e^{\pi i \eta_p}-e^{\pi i \eta_q})^2\ee
where $\eta_p\in Img(\xi)$, $p=1,2,...,k$. For simplicity, we will denote $a_p= e^{\pi i \eta_p}$, where $p=1,2,...,k$ and $|a_p|=1$.

We claim that if $\xi\in\Pi^{k*}_{ne}(\Lambda)$, then $\tau \in P^{\Pi^k(\Lambda), \xi}_2$. To see this we have to first observe that the solution of the equation $A_{\xi,k}^{\Pi^k(\Lambda)}X_\xi= B_{\xi, k}^{\Pi^k(\Lambda)}$ is given by the following elementary symmetric polynomials of $a_1, a_2, \cdots a_k$ defined as
$$e_1(\xi)=(-1)^ka_1 a_2\cdots a_k$$
$$e_2(\xi)=(-1)^{k-1}\sum_{i=1}^k \prod_{j\neq i}a_j$$
$$\cdots$$
$$\cdots$$
$$e_{k-1}(\xi)=\sum_{1\leq i<j \leq k} a_i a_j$$
$$e_k(\xi)=-\sum_{i=1}^k a_i.$$

By hypothesis each $e_i$ belongs to $P^{\Pi^k(\Lambda), \xi}_2$, for $i=0,1,2,\cdots,k$.
As $\tau$ is a symmetric polynomials of $a_1, a_2,\cdots, a_k$, by fundamental theorem of symmetric polynomials it can be written as a polynomial of $e_k, e_{k-1}, \cdots, e_1$. Therefore,  $\tau\in P^{\Pi^k(\Lambda), \xi}_2$. This proves our claim.

Now, we are in a position to prove the lemma. Let $\xi_0$ be a point in $I\cap\Pi^{k*}_{ne}(\Lambda)$. Clearly $\tau(\xi_0)\neq 0$. As $\tau\in P^{\Pi^k(\Lambda), \xi_0}_2$, we can extend $\tau$ continuously on a subinterval $J\subset I$ containing $\xi_0$ so that $\tau(\xi)\neq 0$ for all $\xi\in J$. We will prove that $J\cap \Pi(\Lambda)\subset \bigcup_{j=k}^n \Pi^{k}(\Lambda).$ Suppose there is a point $\xi'\in J\cap\Pi(\Lambda)$ which belongs to $\Pi^{k'}(\Lambda)$ for some $k'<k$. As $J\cap\Pi^{k*}_{ne}(\Lambda)$ is dense in $J\subset I$, we can find a sequence $\{\xi_l\}$ which converges to $\xi'$. Let $\eta^l_i\in Img(\xi_l)$, for $i=1,2,\cdots, k$. Since, $0\leq \eta_i^l<2$, we can find convergent subsequences $\{\eta^{l'}_i\}$ of $\eta^l_i$ and let $\eta^{l'}_i\rightarrow \eta_i'$ as $l'\rightarrow \infty$, for each  $i=1,2,\cdots, k$. As $\Lambda$ is closed, $\eta_i'\in Img(\xi')$, for each $i=1,2,\cdots, k$. This implies $\eta_i'=\eta_j'$ for some $i\neq j$. Hence, from the continuity of the exponential function, we have $\tau(\xi')=0$, which contradicts the fact that $\tau\neq 0$ on $J$. Therefore, the lemma is proved.

\end{proof}

\begin{lem}\label{lem:nec1} There does not exist any subinterval $J$ of $I_0$ such that $\Pi(\Lambda)\cap J\subset \Pi^{k*}_{ne}(\Lambda)$ for any $k=1,2,\cdots n-1$.
\end{lem}
\begin{proof}Suppose the result is not true. Then we have two possibilities.

\underline{Case-I}: Let there exists a subinterval $J$ of $I_0$ such that $\Pi(\Lambda)\cap J\subset \Pi^{1*}_{ne}(\Lambda)$. Consider a point $\xi_0\in J\cap \Pi(\Lambda)$. As $\xi_0\in \Pi^{1*}_{ne}(\Lambda)$, there exist an interval $I_{\xi_0}\subset J$ containing $\xi_0$ and $\varphi_i\in L^1(\R)$, $i=1,2,\cdots,n-1$ such that
\be\label{eq:nec}\hat{\varphi_1}+ \hat{\varphi_2}\chi_0+\hat{\varphi_3}\chi_0^2+\cdots + \hat{\varphi_{n-1}}\chi_0^{n-2}+\chi_0^{n-1}=0\ee
on $I_{\xi_0}\cap \Pi^1(\Lambda)$. Now we will choose a non-zero function $f_n\in L^1(\R)$ whose Fourier transform is supported on $I_{\xi_0}$. After choosing $f_n$ one can construct $f_{i-1}= f_n*\varphi_i$, for $i=1,2,\cdots, n$. Let $\mu$ be the measure which is constructed using the above defined $f_i$ so that the Fourier transform of $\mu$ satisfies equation (\ref{eq:mu}). Then, by \ref{eq:nec} $\hat{\mu}|_{I_{\xi_0}\cap \Pi^1(\Lambda)}=0$. As, $\Pi(\Lambda)\cap J\subset \Pi^{1*}_{ne}(\Lambda)$, $\hat{\mu}$ vanishes on $I_{\xi_0}\cap \Pi(\Lambda)$. Since, each $\hat{f_i}$ is supported in $I_{\xi_0}$, we have $\hat{\mu}|_{\Lambda}=0$, which gives a contradiction. Hence, there can not exist any subinterval $J$ of $I_0$ such that $\Pi(\Lambda)\cap J\subset \Pi^{1*}_{ne}(\Lambda)$.

\underline{Case-II}: Let there exists a subinterval $J$ of $I_0$ such that $\Pi(\Lambda)\cap J\subset \Pi^{k*}_{ne}(\Lambda)$, for some $2\leq k\leq n-1$. Again, let us consider a point $\xi_0\in J\cap\Pi(\Lambda)$. By Lemma \ref{lem:subset}, $\xi_0\in \Pi^{k*}_{su}(\Lambda)$. Hence there exists an interval $I_{\xi_0}\subset J$ and functions $\varphi_i\in L^1(\R)$, $i=1,2,\cdots, n$ such that $(\hat{\varphi}_1, \hat{\varphi}_2,\cdots,\hat{\varphi}_{n-1})$ is a non-zero solution of the system of equation $A_{\xi,k}^{\Pi^k(\Lambda)}X_\xi=B_{\xi, k}^{\Pi^k(\Lambda)} \mbox{ on } I_{\xi_0}\cap \Pi^k(\Lambda)$. Now we can construct a  non-zero measure $\mu$ by first choosing a function $f_n$ whose Fourier transform is supported on $I_{\xi_0}$ and the constructing $f_i$ by the same prescription used above for Case-I. Proceeding as in Case-I, we can prove that there can not exist any subinterval $J$ of $I_0$ such that $\Pi(\Lambda)\cap J\subset \Pi^{k*}_{ne}(\Lambda)$, for any $2\leq k\leq n-1$.

Thus the lemma is proved.
\end{proof}

\newpage

\underline{Proof of the necessary part}
\begin{proof} Now we are ready to prove the necessary part of the theorem. It is enough to show that there does not exist any subinterval $J$ of $I_0$ for which $J\cap \Pi(\Lambda)$ is contained in $\bigcup_{k=1}^{n-1}\Pi^{k*}_{ne}(\Lambda).$ We claim that
there does not exist any subinterval $J$ of $I_0$ such that $J\cap \Pi(\Lambda)\subset \bigcup_{j=1}^l \Pi^{\alpha_j*}_{ne} (\Lambda)$ for any $(\alpha_1,\alpha_2,\cdots,\alpha_l)\in\N^n$ satisfying $1\leq \alpha_1<\alpha_2<\cdots<\alpha_l\leq n-1$. We will prove this by induction on $l$. By Lemma \ref{lem:nec1} the result is true for $l=1$. Suppose the result is true for all $l\leq l_0-1<n-1$. Let us consider $(\alpha_1,\alpha_2,\cdots,\alpha_{l_0})\in\N^n$ satisfying $1\leq \alpha_1<\alpha_2<\cdots<\alpha_{l_0}\leq n-1$. Now, suppose there exist an interval $J\subset I_0$ such that $J\cap \Pi(\Lambda)\subset \bigcup_{j=1}^{l_0}\Pi^{\alpha_j*}_{ne}(\Lambda)$. Then $\Pi^{\alpha_{l_0}*}_{ne}(\Lambda)$ is dense in $J\cap \Pi(\Lambda)$, otherwise we will get an interval $J'\subset J \subset I_0$ such that $J'\cap\Pi(\Lambda)\subset \bigcup_{j=1}^{l_0-1}\Pi^{\alpha_k*}_{ne}(\Lambda)$, which will contradict our induction hypothesis. Since, $\Pi^{\alpha_{l_0}*}_{ne}(\Lambda)$ is dense in $J\cap \Pi(\Lambda)$ and $J\cap\Pi(\Lambda)$ is dense in $J$, $\Pi^{\alpha_{l_0}*}_{ne}(\Lambda)$is dense in $J$. By Lemma
\ref{lem:density} the density of $\Pi^{\alpha_{l_0}*}_{ne}(\Lambda)$ implies that there exist a subinterval $J'\subset J$ such that $J'\cap \Pi(\Lambda)\subset\bigcup_{j=l_0}^n \Pi^j(\Lambda)$. It follows that $J'\cap \Pi(\Lambda)\subset \Pi^{\alpha_{l_0}*}_{ne}(\Lambda)$. But by Lemma \ref{lem:nec1} $J'\cap \Pi(\Lambda)$ can not be contained only in $\Pi^{\alpha_{l_0}*}_{ne}(\Lambda)$. This proves our claim. Therefore, $I$ can not be contained in $\bigcup_{j=1}^{n-1} \Pi^{j*}_{ne}(\Lambda)$ and the necessary part of the theorem is proved.

\end{proof}

We will conclude by pointing out that we can actually improve the necessary condition of our main theorem. Infact, we can notice that the explicit expression of the solutions of the system of linear equation $A_{\xi,k}^{\Pi^k(\Lambda)}X_\xi= B_{\xi, k}^{\Pi^k(\Lambda)}$ is only used in Lemma \ref{lem:density} in order to prove that $\tau\in P^{\Pi^k(\Lambda), \xi}_2$. Hence, in the necessary condition, we can replace the sets $\Pi^{k*}_{ne}(\Lambda)$ by the sets defined below:
$$\tilde{\Pi}^{k*}_{ne}(\Lambda)=\Pi^{k*}_{su}(\Lambda)\cap \{\xi\in \Pi^k(\Lambda): \tau\in P^{\Pi^k(\Lambda), \xi}_2\}$$
where $2\leq k\leq n-2$ and $\tau$ is defined by \ref{eq:tau}.
 Note that from the proof of Lemma \ref{lem:density} we have $\Pi^{k*}_{ne}(\Lambda)\subset \tilde{\Pi}^{k*}_{ne}(\Lambda)$.
\begin{center}
{\bf Acknowledgments}
\end{center}
The author is like to thank I.S.I, Kolkata, for their support. The author also wishes to express his gratitude to Prof. Michael Hopkins for giving him justice.

\end{document}